\documentclass[a4paper,12pt]{article}     
\usepackage{amsmath,amscd,amssymb}        
\usepackage{latexsym}                     
\usepackage{theorem}                 

\newtheorem{theorem}{Theorem}
\newtheorem{corollary}{Corollary}
\newtheorem{proposition}{Proposition}

\theorembodyfont{\rmfamily}
\newtheorem{remark}{Remark}

\newenvironment{definition}
{\smallskip\noindent{\bf Definition\/}:}{\smallskip\par}

\newenvironment{proof}{\begin{ProofwCaption}{Proof}}{\end{ProofwCaption}}
\newenvironment{proof*}[1]{\begin{ProofwCaption}{{#1}}}{\end{ProofwCaption}}
\newenvironment{ProofwCaption}[1]%
 {\addvspace\theorempreskipamount \noindent{\it #1.}\rm}%
 {\qed \par \addvspace\theorempostskipamount}
\newcommand{\qedsymbol}{\mbox{$\Box$}}
\newcommand{\qed}{\hfill\qedsymbol}

\newcommand{\CC}{{\mathbb C}}

\newcommand{\RR}{{\mathbb R}}
\newcommand{\ZZ}{{\mathbb Z}}
\newcommand{\DD}{{\mathbb D}}

\newcommand{\TT}{{\mathbb T}}

\newcommand{\calO}{{\cal O}}

\newcommand{\mathcalL}{{\mathcal L}}

\newcommand{\mathcalD}{{\mathcal D}}

\newcommand{\eps}{\varepsilon}

\newcommand{\kk}{\underline{k}}
\newcommand{\mm}{\underline{m}}
\newcommand{\nn}{\underline{n}}

\newcommand{\rk}{{\rm rk}\,}
\newcommand{\ind}{{\rm ind}\,}

\title{Indices of collections of equivariant 1-forms and characteristic numbers}
\author{W.~Ebeling and S.~M.~Gusein-Zade
\thanks{Partially supported by DFG (Mercator fellowship, Eb 102/8--1),
RFBR--13-01-00755, and NSh--5138.2014.1.
AMS 2010 Math. Subject Classification: 14B05, 58E40, 58A10, 57R20.
}
}
\date{}

\begin{document}

\maketitle

\begin{abstract}
If two $G$-manifolds are $G$-cobordant then characteristic numbers corresponding to the fixed point sets (submanifolds) of subgroups of $G$ and to normal bundles to these sets coincide.
We construct two analogues of these characteristic numbers for singular complex $G$-varieties where $G$ is a
finite group. They are defined as sums of certain indices of collections of 1-forms (with values in the spaces of the irreducible representations of subgroups). These indices are generalizations of the GSV-index (for isolated complete intersection singularities) and the Euler obstruction respectively.
\end{abstract}

\section*{Introduction}
Two closed manifolds are (co)bordant if and only if all their Stiefel--Whitney characteristic
numbers coincide. Unitary ($U$-) manifolds, that is manifolds with complex structures
on the stable normal bundles, are generalizations of complex analytic manifolds adapted
to  (co)bordism theory. Two closed $U$-manifolds are (co)bordant if and only if all
their Chern characteristic numbers coincide. (For the facts from cobordism theory see, e.g., \cite{Stong}.)
Analogues of Chern numbers for (compact)
singular complex analytic varieties were considered in \cite{BLMS} and \cite{Chern_obstr}.
In \cite{BLMS} they were defined for compact complex analytic varieties with isolated
complete intersection singularities (ICIS). 
In \cite{Chern_obstr} (other) characteristic numbers were defined for arbitrary compact
analytic varieties. (For non-singular varieties both sets of characteristic numbers
coincide with the usual Chern numbers.) A survey of these and adjacent results can be found in
\cite{Schwerpunkt}.

Each characteristic number from any of the sets is the sum of indices of singular
(or special) points of (generic) collections of 1-forms defined in \cite{BLMS}
and \cite{Chern_obstr}. The indices from \cite{BLMS} are generalizations of
the GSV-index of a 1-form defined initially for vector fields in \cite{GSV} and \cite{SS}.
The indices from \cite{Chern_obstr} are generalizations of the Euler obstruction defined
in \cite{MPh}. Chern characteristic numbers of (closed) complex analytic manifolds
cannot be arbitrary: they satisfy some divisibility conditions.
In \cite{Buryak} it was shown that the characteristic numbers of singular analytic
varieties from \cite{Chern_obstr} can have arbitrary sets of values.

Let us consider varieties (manifolds) with actions of a finite group $G$. Let $V$
be a $U$-manifold (for example, a complex analytic manifold of (complex) dimension $d$, 
that is with real dimension $2d$) with a $G$-action respecting the $U$-structure.
This means that the trivial extension of the $G$-action on the tangent bundle $TV$
to its stabilization preserves the complex structure on it.
For a subgroup $H\subset G$,
let $V^{H}$ be the set of fixed points of the subgroup $H$. The manifold $V^{H}$ is a
$U$-manifold, i.e.\  its stable tangent bundle has a natural complex structure.
The normal bundle of $V^H$ has the structure of a complex vector bundle
and the group $H$ acts on it. Let 
$r(H)$ be the set of the isomorphism classes of the
complex irreducible representations of the group $H$ and let
$r_+(H)\subset r(H)$ be the set of non-trivial ones. For a component $N$ of $V^H$,
the normal bundle of $V$ is the direct sum of vector bundles
of the form $\alpha\otimes\nu_{\alpha}$ where $\alpha$ runs over all the irreducible
non-trivial representations of $H$ and $\nu_{\alpha}$ are complex vector bundles (with trivial $H$-actions).
Moreover $\sum\limits_{\alpha\in r_+(H)}\rk\alpha\cdot\rk\nu_{\alpha}=
(\dim_{\RR}{V} - \dim_{\RR}{V^{H}})/2$.
For a subgroup $H\subset G$ and for a collection
$\underline{n}=\{n_\alpha\vert \alpha\in r_+(H)\}$,
let $V^{\nn}$ be the union of the components of $V^H$ for which the rank of the
corresponding vector bundle $\nu_{\alpha}$ is equal to $n_\alpha$.

Let $V$ be compact of real dimension $2d$, let ${\nn}=\{n_\alpha\vert \alpha\in r_+(H)\}$
be a collection of integers, let $n_{\bf 1}:=d-\sum\limits_{\alpha\in r_+(H)} n_{\alpha}\rk\alpha$,
and let $\nu_{\bf 1}$ be the stable tangent bundle to $V^{\nn}$.
In the sequel, the collection $\nn$ will contain $n_{\bf 1}$ as well.
For a collection of integers
$\kk=\{k_{\alpha,i}\vert \alpha\in r(H), i01, \ldots , s_\alpha\}$ with
$\sum\limits_{\alpha, i} k_{\alpha,i}=n_{\bf 1}$, let 
$c_{H,\nn,\kk}(V,G)=
\langle\prod\limits_{\alpha, i}c_{k_{\alpha,i}}(\nu_\alpha^*), [V^{\nn}]\rangle$
be the corresponding characteristic number ($[V^{\nn}]\in H_{2 n_{\bf 1}}(V^{\nn})$
is the fundamental class of the manifold $V^{\nn}$). One can assume that
$k_{\alpha,i}\le n_{\alpha}$ for all $(\alpha, i)$. 
These characteristic numbers are invariants of the bordism classes of $U$-manifolds with $G$-actions, i.e.\ if two $G$-$U$-manifolds $(V_1, G)$ and
$(V_2, G)$ are (co)bordant, then, for any $H\subset G$, $\nn$ and $\kk$, the corresponding
characteristic numbers $c_{H,\nn,\kk}(V_1,G)$ and $c_{H,\nn,\kk}(V_2,G)$ coincide. One can see that these numbers depend only on the conjugacy class of the subgroup $H$ (modulo the corresponding identification of the sets of representations of conjugate subgroups).

In this paper we construct, for compact complex singular $G$-varieties, two analogues of these characteristic numbers defined as sums of certain indices of collections of 1-forms (with values in the spaces of the irreducible representations of $H$). These indices are generalizations of the 
GSV-indices and the Chern obstructions defined in \cite{BLMS} and in \cite{Chern_obstr} respectively.

\section{Characteristic numbers in terms of indices of collections of 1-forms} \label{sect1}
Let us give a convenient description of the characteristic numbers defined above for complex
analytic manifolds. Let $V$ be a compact complex analytic manifold of (complex) dimension $d$. We keep the notations of the Introduction. For each pair $(\alpha,i)$, let $\omega^{\alpha,i}_j$,
$j=1,\ldots, n_{\alpha}-k_{\alpha,i}+1$, be generic (continuous)
sections of the bundle $\nu_{\alpha}^*$ (i.e.\  linear forms on $\nu_{\alpha}$).
The characteristic number $c_{H,\nn,\kk}(V,G)$ can be computed as the algebraic
(i.e.\  counted with signs) number of points of $V^{\nn}$
where for each pair $(\alpha,i)$ the 1-forms $\omega^{\alpha,i}_1$, \dots,
$\omega^{\alpha,i}_{n_{\alpha}-k_{\alpha,i}+1}$ are linearly dependent.

For arbitrary sections $\omega^{\alpha,i}_j$, $j=1,\ldots, n_{\alpha}-k_{\alpha,i}+1$,
a point $x\in V^{\nn}$ will be called {\em singular} if for each pair $(\alpha,i)$ the 1-forms
$\omega^{\alpha,i}_1$, \dots, $\omega^{\alpha,i}_{n_{\alpha}-k_{\alpha,i}+1}$ are linearly dependent
at this point. Assume that all singular points of a collection $\{\omega^{\alpha,i}_j\}$
are isolated. In this case the characteristic number $c_{H,\nn,\kk}(V,G)$ is the sum of
certain indices of the collection $\{\omega^{\alpha,i}_j\}$ at the singular points.
Let us describe these indices as degrees of certain maps, as intersection numbers,
and, in the case when all the 1-forms $\omega^{\alpha,i}_j$ are complex analytic, as
dimensions of certain algebras.

For some simplicity, assume first that the group $G$ is abelian and therefore all the
representations $\alpha\in r(H)$ are one-dimensional (i.e.\  they are elements of the group
$H^*=\mbox{Hom}(H,\CC^*)$ of characters of $H$). The linear functions $\omega^{\alpha,i}_j$
on $\nu_{\alpha}$ can be regarded as linear functions on $T_xV$ for
$x\in V^{\nn}$ assuming that ${\omega^{\alpha,i}_j}_{\vert \beta\otimes\nu_\beta}=0$
for $\beta\ne\alpha$. The form $\omega^{\alpha,i}_j$ on $T_xV$ satisfies the condition
\begin{equation}\label{equi_condition1}
 \omega^{\alpha,i}_j(gu)=\alpha(g)\omega^{\alpha,i}_j(u)
\end{equation}
for $g\in H$, $u\in T_xV$. The forms $\omega^{\alpha,i}_j$ can be extended to 1-forms
on $V$ satisfying (\ref{equi_condition1}).
If the group $G$ is not assumed to be abelian and thus among the irreducible representations
$\alpha$ of the subgroup $H$ one may have those of dimension higher than $1$, the form
$\omega^{\alpha,i}_j$ should be considered as a 1-form on $V$ with values in the space $E_{\alpha}$ of
the representation $\alpha$ satisfying the same condition (\ref{equi_condition1}) (where
now $\alpha(g)$ is not a number, but the corresponding operator on $E_{\alpha}$).
The latter simply means that the form $\omega^{\alpha,i}_j$ is $H$-equivariant.

For natural numbers $p$ and $q$ with $p\ge q$, let $M_{p, q}$ be the space of $p\times q$
matrices with complex entries, and let $D_{p,q}$ be the subspace of $M_{p, q}$ consisting of
matrices of rank less than $q$. The complement $W_{p,q} = M_{p, q} \setminus D_{p,q}$ is the
Stiefel manifold of $q$-frames (collections of $q$ linearly independent vectors) in $\CC^p$.
The subset $D_{p,q}$ is an irreducible complex analytic subvariety of $M_{p, q}$ of
codimension $p - q + 1$. Therefore $W_{p,q}$ is $2(p - q)$-connected,
$H_{2p-2q +1}(W_{p,q}; \ZZ) \cong \ZZ$, and one has a natural choice of a generator of
this homology group: the boundary of a small ball in a smooth complex analytic slice
transversal to $D_{p,q}$ at a non-singular point.

For collections $\nn=\{n_\alpha\}$ and $\kk=\{k_{\alpha,i}\}$ with $\sum\limits_{\alpha, i} k_{\alpha,i}=n_{\bf 1}$, 
$k_{\alpha,i}\le n_{\alpha}$ for all $(\alpha, i)$, let
$M_{\nn,\kk}:=\prod\limits_{\alpha,i}M_{n_{\alpha},n_{\alpha}-k_{\alpha,i}+1}$,
$D_{\nn,\kk}:=\prod\limits_{\alpha,i}D_{n_{\alpha},n_{\alpha}-k_{\alpha,i}+1}\subset M_{\nn,\kk}$.
The subset $D_{\nn,\kk}$ is an irreducible subvariety of (the affine space) $M_{\nn,\kk}$
of codimension $\sum\limits_{\alpha, i} k_{\alpha,i}(=n_{\bf 1})$, the complement
$W_{\nn,\kk} = M_{\nn,\kk} \setminus D_{\nn,\kk}$ is $(2n_{\bf 1}-2)$-connected,
$H_{2n_{\bf 1}-1}(W_{\nn,\kk};\ZZ) \cong \ZZ$, and one has a natural choice of a generator of
this homology group. This choice defines the degree (an integer) of a map from an
oriented manifold of dimension $2n_{\bf 1}-1$ to $W_{\nn,\kk}$.

Assume that the subgroup $H$ acts on 
$$
\CC^d=\bigoplus\limits_{\alpha\in r(H)}n_{\alpha}E_{\alpha}
=\bigoplus\limits_{\alpha\in r(H)}\bigoplus\limits_{s=1}^{n_\alpha}E_\alpha^{(s)}
$$
by the representation
$\sum\limits_{\alpha\in r(H)}n_{\alpha}\alpha$ ($\sum n_{\alpha}\rk\alpha=d$).
Here $E^{(s)}_\alpha$, $s=1,\ldots,n_\alpha$ are
copies of the space $E_{\alpha}$ of the representation $\alpha$.
(This means that as a germ of an
$H$-set $(\CC^d,0)$ is isomorphic to $(V,x)$ for $x\in V^{\nn}$.) The corresponding
components of a vector from $\CC^d$ will be denoted by $u^{\alpha}_s$ with
$1\le s\le n_{\alpha}$, $u^{\alpha}_s\in E_{\alpha}$. The vector with the components
$u^{\alpha}_s$ will be denoted by $[u^{\alpha}_s]$.
For a collection $\kk=\{k_{\alpha,i}\}$ with $\sum\limits_{\alpha, i} k_{\alpha,i}=n_{\bf 1}$, 
$k_{\alpha,i}\le\ell_{\alpha}$ for all $(\alpha, i)$, let $\omega^{\alpha,i}_j$,
$j=1,\ldots, n_\alpha-k_{\alpha,i}+1$, be
$H$-equivariant (i.e.\  satisfying the condition (\ref{equi_condition1})) 1-forms on
a neighbourhood of the origin in $\CC^d$ with values in the spaces $E_{\alpha}$.
According to Schur's lemma, at a point $p\in n_{\bf 1} E_{\bf 1}\subset\CC^d$ ($\bf 1$ is the trivial
representation of $H$ and $E_{\bf 1}$ is its space: the complex line), the 1-form 
$\omega^{\alpha,i}_{j\vert p}$ vanishes on $\bigoplus\limits_{\beta\ne\alpha}n_{\beta}E_{\beta}$
and on each copy $E_{\alpha}^{(s)}$ ($s=1,\ldots,n_\alpha$) it is the multiplication by
a (complex) number (depending on $p$).
Thus let $\omega^{\alpha,i}_{j\vert p}([u^{\beta}_s])=
\sum\limits_s\psi^{\alpha,i}_{j,s}(p) u^{\alpha}_s$. 
Let $\Psi$ be the map from $E_{\bf 1}^{n_{\bf 1}}=(\CC^n)^H$ to $M_{\nn,\kk}$ which sends a point
$p\in E_{\bf 1}^{n_{\bf 1}}$ to the collection of $n_\alpha\times (n_\alpha-k_{\alpha, i}+1)$ matrices
$$
\left\{
\begin{pmatrix}
\psi^{\alpha,i}_{1,1}(p) & {\cdots} & \psi^{\alpha,i}_{n_\alpha-k_{\alpha,i}+1,1}(p) \\
\vdots & \cdots & \vdots \\
\psi^{\alpha,i}_{1,n_\alpha}(p) & {\cdots} & \psi^{\alpha,i}_{n_\alpha-k_{\alpha,i}+1,n_\alpha}(p) \\
\end{pmatrix}\right\}\,,
$$
whose columns
consist of the components $\psi^{\alpha,i}_{j,s}(p)$ of the $1$-forms $\omega^{\alpha,i}_{j\vert p}$.
Assume that the collection of the forms has no singular points on $V$
outside of the origin (in a neighbourhood of it). This means that $\Psi^{-1}(D_{\nn,\kk})=\{0\}$.

If $\Psi^{-1}(D_{\kk})=\{0\}$, the origin in $E_{\bf 1}^{n_{\bf 1}}\subset\CC^d$ is an isolated
singular point of the collection $\{\omega^{\alpha,i}_j\}$. In this case let us define
the {\em index} $\ind^{H,\nn,\kk}_{\CC^d,0}\{\omega^{\alpha,i}_j\}$ of the singular point $0$
of the collection $\{\omega^{\alpha,i}_j\}$ as the
degree of the map $\Psi_{\vert S^{2n_{\bf 1}-1}}:S^{2n_{\bf 1}-1}\to W_{\nn,\kk}$ or, what is the
same, as the intersection number of the image $\Psi(E_{\bf 1}^{n_{\bf 1}})$ with $D_{\nn,\kk}$
at the origin.

The origin is a {\em non-degenerate} singular point of the collection $\{\omega^{\alpha,i}_j\}$ if the map $\Psi$ is transversal to the variety $D_{\nn,\kk}$ at a non-singular point of it. The index of a non-degenerate singular point is equal to $\pm 1$. If all the forms $\omega^{\alpha,i}_j$ are complex analytic, the index is equal to $+1$.

The following statement is a reformulation of the definition of the Chern classes of
a vector bundle as obstructions to the existence of several linearly independent sections
of the bundle. Let $V$ be a complex analytic manifold of dimension $d$ with a $G$-action,
let $H$ be a subgroup of $G$, let $V^{H}$ be the set of fixed points of the subgroup $H$
and, for a collection ${\nn}=\{n_{\alpha}\}$, let $V^{\nn}$ be the union of
the corresponding components of $V^H$. For $\kk=\{ k_{\alpha,i}\}$, let $\{\omega^{\alpha,i}_j\}$, $j=1, \ldots , n_\alpha- k_{\alpha,i}+1$, be a collection
of equivariant 1-forms on $V$ with values in the spaces $E_{\alpha}$ with only isolated
singular points on $V^{\nn}$. Then the sum of the indices $\ind^{H,\nn,\kk}_{V,P}\{\omega^{\alpha,i}_j\}$
of these points is equal to the characteristic number $c_{H,\nn,\kk}(V,G)$.

Assume that in the situation described above all the 1-forms $\omega^{\alpha,i}_j$ on $(\CC^d,0)$ 
are complex analytic. One has the following statement.

\begin{proposition}
 The index $\ind^{H,\ell,\kk}_{\CC^d,0}\{\omega^{\alpha,i}_j\}$ is equal to the dimension
of the factor-algebra of the algebra $\calO_{\CC^{n_{\bf 1}},0}$ of germs of holomorphic functions on $(\CC^{n_{\bf 1}},0)$
by the ideal generated by the maximal $($i.e.\  of  size
$(n_{\alpha}-k_{\alpha,i}+1)\times(n_{\alpha}-k_{\alpha,i}+1)$$)$ minors of the matrices
$\left(\psi^{\alpha,i}_{j,s}(p)\right)$.
\end{proposition}

The proof repeats the one in \cite{BLMS}.

\section{Equivariant GSV-indices and characteristic numbers} \label{sect2}
Let 
$$
\CC^N=\bigoplus\limits_{\alpha\in r(H)}m_{\alpha}E_{\alpha}
=\bigoplus\limits_{\alpha\in r(H)}\bigoplus\limits_{s=1}^{m_\alpha}E_\alpha^{(s)}
$$
be the complex vector
space with the representation $\sum\limits_{\alpha\in r(H)}m_{\alpha}{\alpha}$ of a
finite group $H$ ($N=\sum m_{\alpha}\dim E_{\alpha}$; $E_{\alpha}^{(s)}$, $s=1,\ldots,m_\alpha$, are
copies of the space $E_{\alpha}$ of the representation $\alpha$, $\mm:=\{ m_\alpha \}$).
Let $f_{\alpha, \ell}$ with $\alpha\in r(H)$, $\ell=1,\ldots, \ell_{\alpha}$, be $H$-equivariant
germs of holomorphic functions (maps) $(\CC^N,0)\to(E_{\alpha},0)$ (i.e.\ 
$f_{\alpha, \ell}(gx)=\alpha(g) f_{\alpha, \ell}(x)$ for $x\in\CC^N$, $g\in H$)
such that $(V,0)=\{f_{\alpha, \ell}=0 \, | \, \alpha\in r(H), \ell=1,\ldots, \ell_{\alpha}\}$ is an
isolated complete intersection singularity in $(\CC^N,0)$ of codimension
$\sum\limits_{\alpha} \ell_{\alpha}\dim E_{\alpha}$.

The differentials $df_{\alpha,\ell}$ of the functions $f_{\alpha,\ell}$ (with values in $E_\alpha$) at a point
$p\in V\cap m_{\bf 1}E_{\bf 1}$ is an $H$-equivariant linear map from $T\CC^N(\cong\CC^N)$ to $E_\alpha$. According to
Schur's lemma $df_{\alpha,\ell}$ vanishes on $\bigoplus\limits_{\beta\ne\alpha}m_{\beta}E_{\beta}$
and is the multiplication by a complex number (depending on the point $p$) on each copy $E_\alpha^{(s)}$, $s=1,\ldots, m_\alpha$.
Let us denote this number by ${\partial_\alpha^s f_{\alpha,\ell}}(p)$. Thus one has
$$
{df_{\alpha,\ell}}_{\vert p}([u_\beta^s])=\sum_s {\partial_\alpha^s f_{\alpha,\ell}}(p)u_\alpha^s\,.
$$

For a collection $\kk=\{k^{\alpha,i}\}$ with $\sum\limits_{\alpha, i} k^{\alpha,i}=m_{\bf 1}-\ell_{\bf 1}$, 
$k^{\alpha}_i\le n_{\alpha}:=m_{\alpha}-\ell_{\alpha}$ for all $(\alpha, i)$, let $\omega^{\alpha,i}_j$,
$j=1,\ldots,n_{\alpha}-k^{\alpha,i}+1$, be $H$-equivariant 1-forms on
a neighbourhood of the origin in $(\CC^N,0)$ with values in the spaces $E_{\alpha}$.
For $p\in E_{\bf 1}^{m_{\bf 1}}\cap V$, let $\omega^{\alpha,i}_{j\vert p}([u^{\alpha}_s])=
\sum\limits_s\psi^{\alpha,i}_{j,s}(p) u^{\alpha}_s$. Let us define a map
$\Psi:E_{\bf 1}^{m_{\bf 1}}\cap V\to M_{\mm,\kk}$ by 
$$
\Psi(p)=\left\{
\begin{pmatrix}{\partial_\alpha^1 f_{\alpha,1}}(p) & {\cdots} & {\partial_\alpha^1 f_{\alpha,\ell_\alpha}}(p) &
\psi^{\alpha,i}_{1,1}(p) & {\cdots} & \psi^{\alpha,i}_{n_\alpha-k_{\alpha,i}+1,1}(p) \\
\vdots & \cdots & \vdots & \vdots & \cdots & \vdots \\
{\partial_\alpha^{m_\alpha} f_{\alpha,1}}(p) & {\cdots} & {\partial_\alpha^{m_\alpha} f_{\alpha,\ell_\alpha}}(p) &
\psi^{\alpha,i}_{1,m_\alpha}(p) & {\cdots} & \psi^{\alpha,i}_{n
_\alpha-k_{\alpha,i}+1,m_\alpha}(p) \\
\end{pmatrix}\right\}\,,
$$
where the right hand side of the equation is the collection of $m_\alpha\times(m_\alpha-k_{\alpha, i}+1)$
matrices $A_{\alpha,i}$ whose first $\ell_\alpha$ columns consist of the components of the differentials
$d f_{\alpha,\ell_\alpha\vert p}$ of the functions $f_{\alpha,\ell_\alpha\vert p}$ and the last
$n_{\alpha}-k_{\alpha,i}+1$ columns
consist of the components $\psi^{\alpha,i}_{j,s}(p)$ of the $1$-forms $\omega^{\alpha,i}_{j\vert p}$.
Assume that the collection of the forms has no singular points on $V$
outside of the origin (in a neighbourhood of it). This means that $\Psi^{-1}(D_{\mm,\kk})=\{0\}$.

\begin{definition}
 The {\em equivariant} GSV-{\em index} $\ind^{H,\nn,\kk}_{V,0}\{\omega^{\alpha,i}_{j}\}$ of the collection
 $\{\omega^{\alpha,i}_{j}\}$ on the ICIS $(V,0)$ is the degree of the map
 $\Psi_{\vert E_{\bf 1}^{m_{\bf 1}}\cap V\cap S_{\eps}^{2N-1}}: E_{\bf 1}^{m_{\bf 1}}\cap V\cap S_{\eps}^{2N-1} \to W_{\nn,\kk}$ (or, what is the same,
the intersection number of the image $\Psi(E_{\bf 1}^{m_{\bf 1}}\cap V)$ with $D_{\nn,\kk}$
at the origin).
\end{definition}

Being an intersection number, the equivariant GSV-index satisfies the law of conservation of number.
This means that if $\widetilde{V}=V_\lambda$ is an ($H$-invariant) deformation of the ICIS $V$ and
a collection $\{\widetilde{\omega}^{\alpha,i}_{j}\}=\{{\omega^{\alpha,i}_{j;\lambda}}\}$ is
a deformation of the collection $\{\omega^{\alpha,i}_{j}\}$
with isolated singular points on $V_\lambda$, then, for $\lambda$ small enough, one has
$$
\ind^{H,\nn,\kk}_{V,0}\{\omega^{\alpha,i}_{j}\}=
\sum_Q \ind^{H,\nn,\kk}_{\widetilde{V},Q}\{\widetilde{\omega}^{\alpha,i}_{j}\},
$$
where the sum on the right hand side runs over all singular points $Q$ of the collection
$\{\widetilde{\omega}^{\alpha,i}_{j}\}$ on $\widetilde{V}$ in a neighbourhood of the origin
(including the singular points of $\widetilde{V}$ itself).

This implies that, for a compact $H$-variety $V$ with only isolated ($H$-invariant) complete
intersection singularities and for a collection of $H$-equivariant forms $\{\omega^{\alpha,i}_{j}\}$
on it with only isolated singular points, the sum 
$$
\sum_Q \ind^{H,\nn,\kk}_{V,Q}\{\omega^{\alpha,i}_{j}\}
$$
of the equivariant GSV-indices of the collection $\{\omega^{\alpha,i}_{j}\}$ over all its singular
points $Q$ does not depend on the collection and is an invariant of $V$. It can be considered as
an analogue of the corresponding Chern number. Moreover, this Chern number is the same
for varieties with only isolated ($H$-invariant) complete intersection singularities from a, say,
one-parameter family of them. In particular, if this family includes (as the generic member)
a smooth variety, the defined Chern number coincides with the usual one.

Assume that all the forms $\omega^{\alpha,i}_{j}$, $\alpha\in r(H)$, $i=1, \ldots, s_\alpha$,
$j=1, \ldots, n_\alpha-k_{\alpha,i}+1$, are complex analytic (in a neighbourhood of the origin in $\CC^N$).
Let $I_{V, \{\omega^{\alpha,i}_{j}\}}$ be the ideal of the ring $\calO_{\CC^{m_{\bf 1}},0}$ generated by
the germs ${f_{1,1}}|_{\CC^{m_{\bf 1}}}$, \dots, ${f_{1,\ell_{\bf 1}}}|_{\CC^{m_{\bf 1}}}$ and by the maximal
(i.e.\  of size $(m_\alpha-k_{\alpha,i}+1)\times(m_\alpha-k_{\alpha,i}+1)$) minors of the matrices
$A_{\alpha,i}$ for all $(\alpha,i)$.

\begin{theorem} One has
 $$
 \ind^{H,\nn,\kk}_{V,0}\{\omega^{\alpha,i}_{j}\}=
 \dim_{\CC} \calO_{\CC^{m_{\bf 1}},0}/I_{V, \{\omega^{\alpha,i}_{j}\}}.
 $$
\end{theorem}

The proof is essentially the same as in \cite{BLMS}.
\section{Equivariant Chern obstructions and characteristic numbers} \label{sect3}
Let the group $H$ act (by a representation) on the affine space $\CC^N$ and let
$(X,0)\subset(\CC^N,0)$ be an $H$-invariant (reduced) germ of a complex analytic variety
of (pure) dimension $n$. 
Let $\pi : \widehat{X} \to X$ be the Nash
transformation of the variety $X\subset B_\eps$ ($B_\eps=B_\eps^{2N}$ is the ball of a sufficiently small
radius $\eps$ around the origin in $\CC^N$) defined as follows. Let $X_{\rm reg}$ be
the set of smooth points of $X$ and let $G(n,N)$ be the
Grassmann manifold
of $n$-dimensional vector subspaces of $\CC^N$. The action of the group $H$ on $\CC^N$ defines an action
on $G(n,N)$ and on the space of the tautological bundle over it as well.
There is a natural map
$\sigma : X_{\rm reg} \to B_\eps \times G(n,N)$ which
sends a point $x\in X_{\rm reg}$ to $(x, T_x X_{\rm reg})$.
The Nash transform $\widehat{X}$ of the variety $X$ is the closure of the image ${\rm Im}\,
\sigma$ of the map $\sigma$ in $B_\eps \times G(n,N)$, $\pi$ is the natural projection.
The Nash bundle $\widehat{T}$ over $\widehat{X}$ is a
vector bundle of rank $n$
which is the pullback of the tautological bundle on the Grassmann manifold
$G(n,N)$. There is a
natural lifting of the Nash transformation to a bundle map from the Nash bundle
$\widehat{T}$ to the restriction of the tangent bundle $T\CC^N$ of $\CC^N$
to $X$. This is an isomorphism of $\widehat{T}$ and $TX_{\rm reg} \subset T\CC^N$
over the non-singular part $X_{\rm reg}$ of $X$.

For a collection of integers $\nn=\{n_\alpha\}$, let $X^{\nn}$ be the closure in $X$ of the set of
points $x\in X_{\rm reg}\cap X^H$ such that the representation of the group $H$ on the tangent space
$T_xX_{\rm reg}$ is $\bigoplus\limits_{\alpha}n_\alpha \alpha$ and let $\widehat{X}^{\nn}$ be the closure
of the same set in $\widehat{X}$. The restriction of the Nash bundle $\widehat{T}$ to
$\widehat{X}^{\nn}$ is the direct sum of the subbundles $\widehat{T}_\alpha$ ($\alpha\in r(H)$) subject
to the splitting of the representation of $H$ on $\widehat{T}$ into the summands corresponding to
different irreducible representations of $H$. Let the (complex) vector bundle $\widehat{\nu}^*_\alpha$ of rank $n_\alpha$ over $\widehat{X}^{\nn}$ be defined by
$$
\widehat{\nu}^*_\alpha:=\mbox{Hom}_H(\widehat{T}, E_\alpha)=\mbox{Hom}_H(\widehat{T}_\alpha, E_\alpha)\,.
$$

Let $\kk=\{k_{\alpha,i}\}$, $\alpha\in r(H)$, $i=1, \ldots, n_\alpha-k_{\alpha,i}+1$,
be a collection of positive integers with
$\sum\limits_{\alpha,i}k_{\alpha,i} =n_{\bf 1}$ and let $\{\omega^{\alpha,i}_j\}$, $j=1,\ldots, n_\alpha-k_{\alpha,i}+1$, be a
collection of $H$-equivariant $1$-forms on $(\CC^N,0)$ with values in the spaces $E_\alpha$ of the representations
($\omega^{\alpha,i}_j(gu)=\alpha(g)\omega^{\alpha,i}_j(u)$ for $g\in H$).

\begin{remark} One can see that all the constructions below are determined by the restrictions of the forms $\omega^{\alpha,i}_j$ to the regular part of the variety $X$.
\end{remark}

Let $\eps>0$ be small enough so that there is a representative $X$ of the germ $(X,0)$
and representatives $\omega^{\alpha,i}_j$ of the germs of 1-forms inside the ball $B_\eps\subset\CC^N$.

\begin{definition}
A point $P\in X^{\nn}$ is called a {\em special} point of the collection $\{\omega^{\alpha,i}_j\}$  of
1-forms on the variety $X$ if there exists a sequence $\{P_m\}$ of points from $X^{\nn}\cap X_{\rm reg}$
converging to $P$ such that the sequence $T_{P_m}X_{\rm reg}$ of the tangent spaces at the points $P_m$ has an
($H$-invariant) limit $L$ as $m \to \infty$ (in the Grassmann manifold of $n$-dimensional vector subspaces
of $\CC^N$) and the restrictions of the 1-forms $\omega^{\alpha,i}_{1}$, \dots,
$\omega^{\alpha,i}_{n_\alpha-k_{\alpha, i}+1}$ to the subspace
$L\subset T_P\CC^N$ are linearly dependent for each pair $(\alpha, i)$.
\end{definition}

\begin{definition}
The collection $\{\omega^{\alpha,i}_j\}$ of 1-forms has an {\em isolated special point} on the germ $(X,0)$
if it has no special points on $X$ in a punctured neighbourhood of the origin.
\end{definition}

\begin{remark}
On a smooth variety the notions ``special point'' and ``singular point" coincide.
(On a singular variety these notions are different: see \cite[Remark 1.2]{Chern_obstr}.)
A singular point of a collection of 1-forms $\{\omega^{\alpha,i}_j\}$ on a variety can
be non-degenerate only if it is a smooth point of the variety and therefore it is a special point as well.
\end{remark}

Let 
$$
\mathcalL^{\nn, \kk}= \prod\limits_{\alpha,i} \prod\limits_{j=1}^{n_\alpha-k_{\alpha,i}+1} 
\mbox{Hom\,}_H(\CC^N,E_\alpha)
$$
be the space of collections of $E_\alpha$-valued linear functions on $\CC^N$ (i.e.\  of  $H$-equivariant 1-forms with constant coefficients).

\begin{proposition}\label{prop1}
There exists an open and dense subset $U \subset \mathcalL^{\nn,\kk}$ such that each collection
$\{\lambda^{\alpha,i}_j\} \in U$ has only isolated special points on $X^{\nn}$ and, moreover, all these points
belong to the smooth part $X^{\nn}\cap X_{\rm reg}$ of the variety $X^{\nn}$ and are non-degenerate.
\end{proposition}

\begin{proof}
Let $Y\subset X^{\nn}\times \mathcalL^{\nn,\kk}$ be the closure of the set of pairs
$(x, \{\lambda^{\alpha,i}_j\})$
where $x\in X^{\nn}\cap X_{\rm reg}$ and the restrictions of the linear functions $\lambda^{\alpha,i}_{1}$, \dots,
$\lambda^{\alpha,i}_{n_\alpha-k_{\alpha, i}+1}$ to the tangent space $T_xX_{\rm reg}$ are linearly
dependent for each pair $(\alpha,i)$.
Let ${\rm pr}_2:Y\to \mathcalL^{\nn,\kk}$ be the projection to the second factor.
One has $\mbox{codim\,} Y= \sum\limits_{\alpha,i} k_{\alpha,i}=n_{\bf 1}$ and therefore
$\dim Y = \dim \mathcalL^{\nn,\kk}$.

Moreover, $Y\setminus ((X^{\nn}\cap X_{\rm reg})\times \mathcalL^{\nn,\kk})$ is a proper subvariety of $Y$
and therefore its dimension is strictly smaller than $\dim \mathcalL^{\nn,\kk}$. A generic point 
$\Lambda=\{\lambda^{\alpha,i}_j\}$ of
the space $\mathcalL^{\nn,\kk}$ is a regular value of the map ${\rm pr}_2$ which means that it has only
finitely many preimages, all of them belong to $(X^{\nn}\cap X_{\rm reg})\times \mathcalL^{\nn,\kk}$
and the map ${\rm pr}_2$ is non-degenerate at them.
Therefore $\Lambda\times (X^{\nn}\cap X_{\rm reg})$ intersects $Y$ transversally.
This implies the statement.
\end{proof}

\begin{corollary} \label{cor1}
Let $\{\omega^{\alpha,i}_j\}$ be a collection of 1-forms on $X$ with an isolated special
point at the origin. Then there exists a deformation $\{\widetilde \omega^{\alpha,i}_j\}$
of the collection $\{\omega^{\alpha,i}_j\}$ whose special points lie in $X^{\nn}\cap X_{\rm reg}$ and are
non-degenerate. Moreover, as such a deformation one can use $\{\omega^{\alpha,i}_j+
\lambda^{\alpha,i}_j\}$ with a generic collection $\{\lambda^{\alpha,i}_j\}\in \mathcalL^{\nn,\kk}$
small enough.
\end{corollary}

Let
$$
\widehat{\TT}^{\nn,\kk}=\bigoplus_{\alpha,i}\bigoplus_{j=1}^{n_\alpha-k_{\alpha,i}+1}
\widehat{\nu}^*_{\alpha,i,j}\,,
$$
where $\widehat{\nu}^*_{\alpha,i,j}$ are copies of the vector bundle $\widehat{\nu}^*_\alpha$
numbered by $i$ and $j$.
Let ${\widehat\DD}^{\nn,\kk}\subset{\widehat\TT}^{\nn,\kk}$ be the set of pairs $(x,\{\eta^{\alpha, i}_j\})$,
$x\in \widehat{X}^{\nn}$, $\eta^{\alpha, i}_j\in \widehat{\nu}^*_{\alpha,i,j\vert x}$,
such that $\eta^{\alpha, i}_1$, \dots, $\eta^{\alpha, i}_{n_\alpha-k_{\alpha, i}+1}$ are linearly
dependent for each pair $(\alpha,i)$.

The collection $\{\omega^{\alpha,i}_j\}$ defines a section $\widehat{\omega}$ of the bundle
$\widehat{\TT}^{\nn,\kk}$. The image of this section does not intersect $\widehat\DD^{\nn,\kk}$ outside of
the preimage $\pi^{-1}(0)\subset\widehat{X}^{\nn}$ of the origin.
The
map $\widehat{\TT}^{\nn,\kk}\setminus \widehat{\DD}^{\nn,\kk}\to \widehat{X}^{\nn}$ is a fibre bundle.
The fibre
$W_x=\widehat{\TT}^{\nn,\kk}_x \setminus \widehat{\DD}^{\nn,\kk}_x$ of it is $(2n_{\bf 1}-2)$-connected, its homology group
$H_{2n_{\bf 1}-1}(W_x;\ZZ)$ is isomorphic to $\ZZ$ and has a natural generator.
(This follows from the fact that $W_x$ is homeomorphic to
$M_{\nn,\kk}\setminus D_{\nn,\kk}$ from Section \ref{sect1}.)
The latter fact implies that the fibre bundle $\widehat{\TT}^{\nn,\kk}\setminus
\widehat{\DD}^{\nn,\kk}\to \widehat{X}^{\nn}$ is homotopically simple in dimension $2n_{\bf 1}-1$,
i.e.\ the fundamental group $\pi_1(\widehat{X}^{\nn})$ of the base acts trivially on the homotopy group
$\pi_{2n_{\bf 1}-1}(W_x)$ of the fibre, the last one being isomorphic to the homology
group $H_{2n_{\bf 1}-1}(W_x;\ZZ)$: see, e.g., \cite{St}. 

\begin{definition}
The {\em local Chern obstruction (index)} ${\rm Ch}^{H,\nn,\kk}_{X,0}\,\{\omega^{\alpha,i}_j\}$ of the collections
of germs of 1-forms $\{\omega^{\alpha,i}_j\}$ on $(X,0)$ at the origin is the (primary, and in fact
the only) obstruction to extend the section $\widehat{\omega}$ of the fibre bundle
$\widehat{\TT}^{\nn,\kk}\setminus\widehat{\DD}^{\nn,\kk}\to \widehat{X}^{\nn}$ from the preimage of
a neighbourhood of the sphere $S_\eps= \partial B_\eps$ to $\widehat{X}^{\nn}$, more precisely its value
(as an element of $H^{2n_{\bf 1}}(\pi^{-1}(X^{\nn}\cap B_\eps), \pi^{-1}(X^{\nn} \cap S_\eps);\ZZ)$\,) on the
fundamental class of the pair $(\pi^{-1}(X^{\nn}\cap B_\eps), \pi^{-1}(X^{\nn} \cap S_\eps))$.
\end{definition}

The definition of the local Chern obstruction ${\rm Ch}^{H,\nn,\kk}_{X,0}\,\{\omega^{\alpha,i}_j\}$ 
can be reformulated in the following way. The collection of 1-forms $\{\omega^{\alpha,i}_j\}$
defines also a section $\check{\omega}$ of the trivial bundle $X^{\nn}\times \mathcalL^{\nn,\kk}\to X^{\nn}$, namely,
$\check{\omega}: X^{\nn}\to X^{\nn}\times \mathcalL^{\nn,\kk}$. Let
$\mathcalD^{\nn,\kk}\subset X^{\nn}\times \mathcalL^{\nn,\kk}$
be the closure of the set of pairs $(x, \{\lambda^{\alpha,i}_j\})$ such that $x \in X^{\nn}\cap X_{\rm reg}$
and the restrictions of the linear functions $\lambda^{\alpha,i}_1$, \dots ,
$\lambda^{\alpha,i}_{n_\alpha-k_{\alpha,i}+1}$ to
$T_xX_{\rm reg} \subset \CC^N$ are linearly dependent for each pair $(\alpha, i)$.
For $x\in X^{\nn}\cap S_\eps$ ($\eps$ small enough), $\check{\omega}\notin \mathcalD^{\nn,\kk}$
and the local Chern obstruction is the value on the fundamental class of the pair 
$(X^{\nn}\cap B_\eps,X^{\nn}\cap S_\eps)$ of the first obstruction to
extend the map 
$\check{\omega}: X^{\nn}\cap S_\eps\to(X^{\nn}\times \mathcalL^{\nn,\kk})\setminus\mathcalD^{\nn,\kk}$
to a map $X^{\nn}\cap B_\eps\to(X^{\nn}\times \mathcalL^{\nn,\kk})\setminus\mathcalD^{\nn,\kk}$.
(This follows, e.g.\ from the fact that, due to Corollary~\ref{cor1}, the collection $\{\omega^{\alpha,i}_j\}$ can be deformed in such a way that all the special points of the deformed collection lie in the regular part of $X^{\nn}$ and thus the corresponding obstructions on $X^{\nn}$ and on $\widehat{X}^{\nn}$ coincide.)
The map $\check{\omega}$ is also defined on $\CC^N$ (a section of the trivial bundle $\CC^N\times \mathcalL^{\nn,\kk}\to\CC^N$):
$\check{\omega}: \CC^N\to \CC^N\times \mathcalL^{\nn,\kk}$.
For $x\in S_\eps$ ($\eps$ small enough), $\check{\omega}\notin \mathcalD^{\nn,\kk}
\subset X^{\nn}\times \mathcalL^{\nn,\kk}\subset \CC^N\times \mathcalL^{\nn,\kk}$
and the local Chern obstruction is the value on the fundamental class of the pair 
$(B_\eps,S_\eps)$ of the first obstruction to
extend the map 
$\check{\omega}: S_\eps\to(\CC^N\times \mathcalL^{\nn,\kk})\setminus\mathcalD^{\nn,\kk}$
to a map $B_\eps\to(\CC^N\times \mathcalL^{\nn,\kk})\setminus\mathcalD^{\nn,\kk}$.
This means that it is equal to the intersection number $(\check{\omega}(\CC^N) \circ \mathcalD^{\nn\kk})_0$
at the origin in $\CC^N \times \mathcalL^{\bf k}$.

Being a (primary) obstruction, the local Chern obstruction satisfies the law of conservation of
number, i.e.\ if a collection of 1-forms $\{\widetilde\omega^{\alpha, i}_j\}$ is a deformation of the
collection $\{\omega^{\alpha,i}_j\}$ with only isolated special points on $X$, then 
$$
{\rm Ch}^{H,\nn,\kk}_{X,0}\,\{\omega^{\alpha,i}_j\}
= \sum_{Q}{\rm Ch}^{H,\nn,\kk}_{X,Q}\, 
\{\widetilde\omega^{\alpha,i}_j\}\,,
$$
where the sum on the right hand side is over all special points $Q$
of the collection $\{\widetilde\omega^{\alpha,i}_j\}$ on $X$ in a neighbourhood of the origin.

Along with Corollary~\ref{cor1} this implies the following statements.

\begin{proposition}\label{prop2}
The local Chern obstruction ${\rm Ch}^{H,\nn,\kk}_{X,0}\,\{\omega^{\alpha,i}_j\}$ of a collection
$\{\omega^{\alpha,i}_j\}$ of germs of 1-forms is equal to the algebraic (i.e.\  counted with signs)
number of special points on
$X$ of a generic deformation of the collection.
\end{proposition}

If all the 1-forms $\omega^{\alpha,i}_j$ and their generic deformations are holomorphic,
the local Chern obstructions of the deformed collection at its special points are equal to $1$
and therefore the local Chern obstruction ${\rm Ch}^{H,\nn,\kk}_{X,0}\,\{\omega^{\alpha,i}_j\}$
is equal to the number of special points of the deformation.

\begin{proposition}\label{prop3}
If a collection $\{\omega^{\alpha,i}_j\}$ ($\alpha\in r(H)$, $i=1,\ldots, s_\alpha$, $j=1,\ldots, n_\alpha-k_{\alpha,i}+1$) of 1-forms on a compact (say, projective) $G$-variety $X$ has
only isolated special points on $X^H$, then the sum of the local Chern obstructions of the
collection $\{\omega^{\alpha,i}_j\}$ at these points does not depend on the collection and
therefore is an invariant of the variety.
\end{proposition}

One can consider this sum as the corresponding Chern number of the $G$-variety $X$.

Let $(X,0)$ be an isolated $H$-invariant complete intersection singularity (see Section~\ref{sect2}).
As it was described there, a
collection of germs of 1-forms $\{\omega^{\alpha,i}_j\}$ on $(X,0)$ with an isolated special point at the
origin has an index $\ind^{H,\nn,\kk}_{X,0}\, \{\omega^{\alpha,i}_j\}$ which is an analogue of the GSV-index
of a 1-form. The fact that both the Chern obstruction and the index satisfy the law of
conservation of number and they coincide on a smooth manifold yields the following statement.

\begin{proposition}\label{prop4}
For a collection $\{\omega^{\alpha,i}_j\}$ of germs of 1-forms on an isolated $H$-invariant
complete intersection singularity $(X,0)$ the difference
$$
{\rm ind}^{H,\nn,\kk}_{X,0}\, \{\omega^{\alpha,i}_j\} - {\rm Ch}^{H,\nn,\kk}_{X,0} \, \{\omega^{\alpha,i}_j\}
$$
does not depend on the collection and therefore is an invariant of the germ of the variety.
\end{proposition}

Since, by Proposition~\ref{prop1}, ${\rm Ch}^{H,\nn,\kk}_{X,0} \, \{\omega^{\alpha,i}_j\}=0$
for a generic collection
$\{\omega^{\alpha,i}_j\}$ of linear functions on $\CC^N$, one has the following statement.

\begin{corollary}
One has
$$
{\rm Ch}^{H,\nn,\kk}_{X,0} \, \{\omega^{\alpha,i}_j\} = {\rm ind}^{H,\nn,\kk}_{X,0}\, \{\omega^{\alpha,i}_j\}
- {\rm ind}^{H,\nn,\kk}_{X,0}\,
\{\lambda^{\alpha,i}_j\}
$$
for a generic collection $\{\lambda^{\alpha,i}_j\}$ of $H$-equivariant linear functions on $\CC^N$.
\end{corollary}

\begin{remark}
Instead of working with $E_\alpha$-valued 1-forms on a compact variety $X$, one can also consider 1-forms with values ''in local systems of coefficients'', i.e.\ in ($H$-)vector bundles of the form $E_\alpha \otimes L_\alpha$ where $L_\alpha$ is a (usual) line bundle over $X$. The local indices (both GSV- and Chern ones) of a collection of 1-forms are defined in the same way and the sums of these indices give invariants of a compact $G$-variety $X$ with the bundles $L_\alpha$ which can be regarded as analogues of the Chern numbers $\langle\prod\limits_{\alpha, i}c_{k_{\alpha,i}}(\nu_\alpha^* \otimes L_\alpha), [X]\rangle$ (see the Introduction).
\end{remark}

\begin{remark} The definitions and constructions of this section work for $G$ being a compact Lie group as well.
\end{remark}

\bigskip
\noindent Leibniz Universit\"{a}t Hannover, Institut f\"{u}r Algebraische Geometrie,\\
Postfach 6009, D-30060 Hannover, Germany \\
E-mail: ebeling@math.uni-hannover.de\\

\medskip
\noindent Moscow State University, Faculty of Mechanics and Mathematics,\\
Moscow, GSP-1, 119991, Russia\\
E-mail: sabir@mccme.ru

\end{document}